\documentclass[11pt]{article}
\usepackage{enumerate}
\usepackage{graphicx}          
\usepackage{graphics}
\usepackage{epsfig} 
\usepackage{psfrag}
\usepackage[cmex10]{amsmath}\interdisplaylinepenalty=2500
\usepackage{amssymb}
\usepackage{amsthm}

\newcommand{\trans}{{}^\top}
\renewcommand{\Re}{\mathbb R}

\newcommand{\cX}{\mathcal{X}}
\newcommand{\cU}{\mathcal{U}}
\newcommand{\cY}{\mathcal{Y}}
\newcommand{\cN}{\mathcal{N}} 

\newcommand{\cO}{\mathcal{O}}

\theoremstyle{definition}

\newtheorem{thm}{Theorem}[section]
\newtheorem{prop}[thm]{Proposition}%[section]
\newtheorem{defn}[thm]{Definition}%[section]
\newtheorem{lem}[thm]{Lemma}%[section]
\newtheorem*{rem}{Remark}
\newtheorem*{example}{Example}
\newtheorem*{problem}{Set Stabilization Problem}

\newcommand{\phiu}{\phi_u}
\newcommand{\LL}{L^+}

\newcommand{\Lu}{L^+_u}

\newcommand{\Ju}{J^+_u}
\newcommand{\JJ}{J^+}
\renewcommand{\SS}{{\cal S}}

\DeclareMathOperator{\cl}{cl}

\begin{document}
\title{Reduction Principles and the Stabilization of
  Closed Sets for Passive Systems}%

\author{Mohamed~I.~El-Hawwary,
  Manfredi~Maggiore,%
\thanks{The authors are with the Department of Electrical and Computer
  Engineering, University of Toronto, Toronto, ON M5S 3G4,
  Canada. E-mails: {\tt\footnotesize
    \{melhawwary,maggiore\}@control.utoronto.ca}.}%
\thanks{This research was supported by the National Sciences and
  Engineering Research Council of Canada.}%
\\[2ex] {\small This paper appeared in the {\em IEEE Transactions on
  Automatic Control}, vol. 55, no. 4, 2010, pp. 982--987}
}

\maketitle

\begin{abstract}
In this paper we explore the stabilization of closed invariant sets
for passive systems, and present conditions under which a
passivity-based feedback asymptotically stabilizes the goal set. Our
results rely on novel reduction principles allowing one to extrapolate
the properties of stability, attractivity, and asymptotic stability of
a dynamical system from analogous properties of the system on an
invariant subset of the state space.
\end{abstract}

\section{Introduction}
The notion of passivity for state space representations of nonlinear
systems, pioneered by Willems in the early
1970's,~\cite{willems1,willems2}, was instrumental for much research
on nonlinear equilibrium stabilization. Key contributions in this area
were made in the early 1980's by Hill and Moylan
in~\cite{HilMoy76,HilMoy77,HilMoy80_1,HilMoy80_2}, and later by
Byrnes, Isidori, and Willems, in their landmark paper~\cite{ByrIsiWil91}.
More recently, in a number of
papers~\cite{ShiFra00,Shi00_1,Shi00_2}, Shiriaev and Fradkov addressed the
problem of stabilizing compact invariant sets for passive nonlinear
systems. Their work is a direct extension of the equilibrium
stabilization results by Byrnes, Isidori, and Willems in \cite{ByrIsiWil91}.

In this paper we develop a theory of set stabilization for passive
systems which generalizes the equilibrium theory
of~\cite{ByrIsiWil91}, as well as the results
in~\cite{ShiFra00,Shi00_1,Shi00_2}.  We investigate the stabilization
of a closed set $\Gamma$, not necessarily compact, which is open-loop
invariant and contained in the zero level set of the storage
function. Our results answer this question: {\em when is it that a
  passivity-based controller makes $\Gamma$ asymptotically stable for
  the closed-loop system?} Even in the special case when $\Gamma$ is
an equilibrium, our theory yields novel results, among them necessary
and sufficient conditions for the passivity-based asymptotic
stabilization of the equilibrium in question without imposing that the
storage function be positive definite. The theory
in~\cite{ByrIsiWil91}, and~\cite{ShiFra00,Shi00_1,Shi00_2} does not
handle this situation.

The key insight behind the development of the results presented in
this paper is the realization that at the heart of the stabilization
problem by passivity-based feedback there lies a so-called reduction
problem for a dynamical system $\Sigma: \dot x = f(x)$: {\em Consider
  two closed sets $\Gamma$ and $\cO$, with $\Gamma \subset \cO$, which
  are invariant for $\Sigma$; suppose that $\Gamma$ is stable,
  attractive, or asymptotically stable for the restriction of $\Sigma$
  to $\cO$. When is it that $\Gamma$ is stable, attractive, or
  asymptotically stable with respect to the whole state space?}  We
answer this question by presenting three novel reduction principles
for attractivity, stability, and asymptotic stability that have
independent interest and are applicable to other problems in control
theory. The proofs of these and other results are omitted in this
shortened paper. The interested reader is referred to the full
version~\cite{ElHMag09_2} and~\cite{ElHMag09_3}.

{\em Outline:} Section~\ref{sec:prem-prob} presents stability
definitions and reviews the notion of limit set and that of
prolongational limit set. In Section~\ref{sec:problems} we state the
passivity-based set stabilization problem and illustrate its
relationship to the reduction problem. We then present the reduction
principles.  Section~\ref{sec:detectability} presents a novel notion
of detectability using which, in
Section~\ref{sec:solution_set_stabilization}, we solve the
passivity-based set stabilization problem. The main result,
Theorem~\ref{thm:semi-asymptotic_stability}, generalizes previous
results on passivity-based stabilization. This fact is discussed in
Section~\ref{sec:discussion}.

\section{Preliminaries} 
\label{sec:prem-prob}
In this paper we consider control-affine systems described by
\begin{equation}
\begin{aligned}
&\dot{x} =  f(x)+\sum_{i=1}^m g_i(x)u_i \\
&y  =  h(x)
\end{aligned}
\label{eq:sys}
\end{equation}
with state space $\mathcal{X}\subset \Re^n$, set of input values
$\cU\subset \Re^m$ and set of output values $\cY \subset \Re^m$. The
set $\cX$ is assumed to be either an open subset or a smooth
submanifold of $\Re^n$. We assume that $f$ and $g_i$, $i=1,\ldots m$,
are smooth vector fields on $\cX$, and that $h: \cX \to \cY$ is a
smooth mapping.

\subsubsection*{Notation}
Let $\Re^+=[0,+\infty)$.  Given either a smooth feedback $u(x)$ or a
  piecewise-continuous open-loop control $u(t):\Re^+\to \cU$, we
  denote by $\phiu(t,x_0)$ the unique solution of~\eqref{eq:sys} with
  initial condition $x_0$. By $\phi(t,x_0)$ we denote the solution of
  the open-loop system $\dot x = f(x)$ with initial condition
  $x_0$. Given an interval $I$ of the real line and a set $S \in \cX$,
  we denote by $\phiu(I,S)$ the set $\phiu(I,S):=\{\phiu(t,x_0): t \in
  I, x_0 \in S\}$. The set $\phi(I,S)$ is defined analogously.  Given
  a closed nonempty set $S \subset \Re^n$, a point $x \in \Re^n$, and
  a vector norm $\|\cdot\|:\Re^n\to \Re$, the point-to-set distance
  $\|x\|_S$ is defined as $\|x\|_S :=\inf\{\|x-y\|: y \in S\}$. The
  state space $\cX$, being a subset of $\Re^n$, inherits a norm from
  $\Re^n$, which we will denote $\|\cdot\|: \cX \to \Re$. For a
  constant $\alpha > 0$, a point $x \in \cX$, and a set $S \subset
  \cX$, define the open sets $B_{\alpha}(x)=\{y \in \cX: \|y-x\|<
  \alpha\}$ and $B_{\alpha}(S)=\{y \in \cX: \|y\|_{S}< \alpha\}$. We
  denote by $\cl(S)$ the closure of the set $S$, and by $\cN(S)$ an
  open neighbourhood of $S$, that is, an open subset of $\cX$
  containing $S$.  We use the standard notation $L_f V$ to denote the
  Lie derivative of a $C^1$ function $V$ along a vector field $f$.

\subsubsection*{Passivity}
Throughout this paper it is assumed that \eqref{eq:sys} is
passive with smooth nonnegative storage function $V: \cX \to \Re$,
i.e., $V$ is a $C^r$ ($r \geq 1$) nonnegative function such that, for
all piecewise-continuous functions $u: [0, \infty) \to \cU$, for all
  $x_0 \in \mathcal{X}$, and for all $t$ in the maximal interval of
  existence of $\phiu(\cdot,x_0)$,
\begin{equation*}
V(\phiu(t,x_0))-V(x_0) \leq \int_0^t {u(\tau)\trans y(\tau)} d\tau
\label{eq:storage_inequality},
\end{equation*}
where $y(t) = h(\phiu(t,x_0))$. It is well-known (see \cite{HilMoy76})
that the passivity property above is equivalent to the two conditions
\begin{equation}\label{eq:passivity}
(\forall x \in \cX) \ L_f V(x) \leq 0 \ \mbox{and} \
\ L_g V(x) = h(x)\trans,
\end{equation}
where $L_g V$ denotes the row vector $[L_{g_1} V \ \cdots \ L_{g_m}
  V]$.

\subsubsection*{Set stability and attractivity}
All definitions below are standard and can be found
in~\cite{BatSze67}. Let $\Gamma \subset \cX$ be a closed positively
invariant for a dynamical system
\begin{equation}
\label{eq:Sigma}
\Sigma: \dot x = f(x),\ \  x \in \cX.
\end{equation}

\begin{defn} [Set stability and attractivity]
\label{defn:set-stab}

\begin{enumerate}[(i)]

\item $\Gamma$ is {\em stable} for $\Sigma$ if for all
  $\varepsilon >0$ there exists a neighbourhood $\cN(\Gamma)$ such
  that $\phi(\Re^+,\cN(\Gamma)) \subset B_\varepsilon(\Gamma)$.

\item $\Gamma$ is an {\em attractor} for $\Sigma$ if there exists a
  neighbourhood $\mathcal{N}(\Gamma)$ such that, for all $x_0 \in
  \cN(\Gamma)$, $\lim_{t \to \infty}\|\phi(t,x_0)\|_{\Gamma}=0$.

\item $\Gamma$ is a {\em global attractor} for $\Sigma$ if it is 
an attractor with $\cN(\Gamma) = \cX$.

\item $\Gamma$ is {\em [globally] asymptotically stable} for
  $\Sigma$ if it is stable and attractive [globally attractive]
  for $\Sigma$.

\end{enumerate}
\end{defn}

%\begin{rem}
If $\Gamma$ is a compact positively invariant set, then the concepts
of stability, attractivity, and asymptotic stability, as defined
above, are equivalent to the familiar $\epsilon$-$\delta$ notions of
uniform stability, attractivity, and asymptotic stability found, e.g.,
in~\cite[Definition 8.1]{Kha02}.  In the unbounded case, however, our
definitions of attractivity and asymptotic stability, referred to as
semi-attractivity and semi-asymptotic stability in~\cite{BatSze67},
are weaker than the corresponding $\epsilon$-$\delta$ notions.  For
instance, the $\epsilon-\delta$ notion of attractivity requires that
the domain of attraction of $\Gamma$ contains a tube of radius
$\delta$, whereas the notion of attractivity in the definition above
does not, and in fact if $\Gamma$ is unbounded the width of its domain
of attraction may shrink to zero at infinity.
%\end{rem}

\begin{defn}[Relative set stability and attractivity]
Let $\cO \subset \cX$ be such that $\cO \cap \Gamma \neq \emptyset$.
We say that $\Gamma$ is {\em stable relative to} $\cO$ for $\Sigma$
if, for any $\varepsilon >0$, there exists a neighbourhood
$\cN(\Gamma)$ such that $\phi(\Re^+,\cN(\Gamma) \cap \cO)
  \subset B_\varepsilon(\Gamma)$. Similarly, one modifies all other
  notions in Definition~\ref{defn:set-stab} by restricting initial
  conditions to lie in $\cO$.
\end{defn}

\begin{defn}[Local stability and attractivity near a set]
\label{defn:local_stab}
Let $\Gamma$ and $\cO$, $\Gamma \subset \cO \subset \cX$, be
positively invariant sets. The set $\cO$ is {\em locally stable near
  $\Gamma$} if for all $x \in \Gamma$, for all $c>0$, and all
$\varepsilon>0$, there exists $\delta>0$ such that for all $x_0 \in
B_\delta(\Gamma)$ and all $t>0$, whenever $\phi([0,t],x_0) \subset
B_c(x)$ one has that $\phi([0,t],x_0) \subset B_\varepsilon(\cO)$.
The set $\cO$ is {\em locally attractive near $\Gamma$} if
there exists a neighbourhood $\cN(\Gamma)$ such that, for all $x_0 \in
\cN(\Gamma)$, $\phi(t,x_0) \to \cO$ at $t \to +\infty$.
\end{defn}
\begin{figure}[htb]
\psfrag{bdg}{\footnotesize $B_\delta(\Gamma)$}
\psfrag{beo}{\footnotesize $B_\varepsilon(\cO)$}
\psfrag{G}{\footnotesize $x \in\Gamma$}
\psfrag{O}{\footnotesize $\cO$}
\psfrag{bcx}{\footnotesize $B_c(x)$}
\centerline{\includegraphics[width=.4\textwidth]{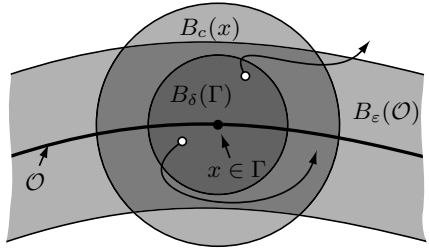}}
\caption{An illustration of the notion of local stability near $\Gamma$}
\label{fig:local_stability}
\end{figure}
The property of local stability can be rephrased as follows. Given an
arbitrary ball $B_c(x)$ centred at a point $x$ in $\Gamma$,
trajectories originating in $B_c(x)$ sufficiently close to $\Gamma$
cannot travel far away from $\cO$ before first exiting $B_c(x)$; see
Figure~\ref{fig:local_stability}. It is immediate to see that if
$\Gamma$ is stable, then $\cO$ is locally stable near $\Gamma$.
\begin{defn}[Local uniform boundedness]
\label{defn:LUB}
The system $\Sigma$ is {\em locally uniformly bounded near
  $\Gamma$} if for each $x \in \Gamma$ there exist positive scalars
$\lambda$ and $m$ such that $\phi(\Re^+,B_\lambda(x)) \subset B_m(x)$.
\end{defn}
\subsubsection*{Limit Sets}

In order to characterize the asymptotic properties of bounded
solutions, we will use the well-known notion of limit set, due to
G.~D.~Birkhoff (see~\cite{Bir27}), and that of prolongational limit
set, due to T.~Ura (see~\cite{Ura59}).  Given a smooth feedback $u(x)$
and a point $x_0 \in \cX$, the {\em positive limit set} (or
$\omega$-limit set) of the closed-loop solution $\phiu(t,x_0)$ is
defined as
\[
\begin{aligned}
\Lu(x_0) := \{p \in \cX: (\exists \{t_n\}\subset \Re^+)  \, & t_n \to
+ \infty,\\
&\phiu(x_0,t_n)\to p\}.
\end{aligned}
\]
The positive limit set of the open-loop solution $\phi(t,x_0)$,
defined in an analogous way, is denoted $\LL(x_0)$.  We let
$\Lu(S):=\bigcup_{x_0 \in S} \Lu(x_0)$ and $\LL(S):=\bigcup_{x_0 \in
  S} \LL(x_0)$.

If $U\subset \cX$ and $x_0 \in \cl(U)$, the {\em prolongational limit
  set relative to $U$} of an open-loop solution $\phi(t,x_0)$ is
defined as
\[
\begin{aligned}
\JJ(x_0,U):=\{ p \in {\mathcal X} : (\exists \{(x_n,&t_n)\} \subset U
\times \Re^+), \, x_n \to x_0,\\
& t_n \to +\infty, \,
\phi(x_n,t_n)\to p\}.
\end{aligned}
\]
We denote $\JJ(S,U):=\bigcup_{x_0 \in S} \Ju(x_0,U)$.  One can show
that if $x_0 \in \cl(U)$, then $\LL(x_0) \subset \JJ(x_0,U)$.

\section{Stabilization problem and reduction
  principles}\label{sec:problems}  

The main objective of this paper is the stabilization of a closed set
$\Gamma$ using {\em passivity-based feedbacks} of the form
\begin{equation}
\label{eq:passivity-based_feedback}
u=-\varphi(x), \text{ with } \varphi(\cdot)\Big|_{h(x)=0}=0,
\ h(x)\trans \varphi(x)\Big|_{h(x)\neq 0} >0,
\end{equation}
where $\varphi: \cX \to \cU$ is a smooth function.  The class of
passivity-based feedbacks in~\eqref{eq:passivity-based_feedback}
includes that of output feedback controllers $u=-\varphi(h(x))$
commonly used in the literature on passive systems.

\begin{problem}
Given a closed set $\Gamma \subset V^{-1}(0)= \{x \in \cX: V(x)=0\}$
which is positively invariant for the open-loop system in
\eqref{eq:sys}, and given a passivity-based feedback of the
form~\eqref{eq:passivity-based_feedback}, find conditions guaranteeing
that $\Gamma$ is [globally] asymptotically stable for the
closed-loop system.
\end{problem}
The rationale behind passivity-based feedback is the following.
Using~\eqref{eq:passivity} and the properties of the passivity-based
feedback~\eqref{eq:passivity-based_feedback}, the time derivative of
the storage function $V$ along trajectories of the closed-loop system
formed by~\eqref{eq:sys} with
feedback~\eqref{eq:passivity-based_feedback} is given by
\begin{equation}\label{eq:Vdot}
\begin{aligned}
\frac{d V(\phiu(t,x_0))}{dt} =& L_f V(\phiu(t,x_0))\\ &- L_g V
(\phiu(t,x_0)) \varphi(\phiu(t,x_0))\\
\leq & - h(\phiu(t,x_0))\trans \varphi(\phiu(t,x_0)) \leq 0.
\end{aligned}
\end{equation}
Thus, a passivity-based feedback renders the storage function $V$
nonincreasing along solutions of the closed-loop system. One expects
that if the system enjoys suitable properties, then the storage
function should decrease asymptotically to zero and the solutions
should approach a subset of $V^{-1}(0)$, hopefully the set
$\Gamma$.

Our point of departure in understanding what system properties yield
the required result is the well-known property, found in the proof of
Theorem~3.2 in~\cite{ByrIsiWil91}, that, for all $x_0 \in \cX$, the
positive limit set $\Lu(x_0)$ of the closed-loop system is invariant
for the open-loop system and such that $\Lu(x_0) \subset h^{-1}(0)$.
Let $\cO$ denote the {\em maximal} set contained in $h^{-1}(0)$ which
is invariant for the open-loop system.  In light of the property
above, if $\Lu(x_0)$ is non-empty, then it must be contained in
$\cO$. Therefore, all bounded trajectories of the closed-loop system
asymptotically approach $\cO$.  Since $L_f V\leq 0$, $V$ is
nonincreasing along solutions of the open-loop system, and so
$V^{-1}(0)$ is an invariant set for the open-loop system. Moreover,
since $V$ is nonnegative, any point $x \in V^{-1}(0)$ is a local
minimum of $V$ and hence $dV(x)=0$. Therefore, $L_g V(x)=h(x)\trans=0$
on $V^{-1}(0)$, and so $\Gamma \subset V^{-1}(0) \subset
h^{-1}(0)$. Since $V^{-1}(0)$ is invariant and contained in
$h^{-1}(0)$, it is necessarily a subset of $\cO$ (this implies that
$\cO$ is not empty). Putting everything together, we conclude that
\begin{equation}
\label{eq:inclusion}
\Gamma \subset V^{-1}(0) \subset \cO \subset h^{-1}(0).
\end{equation}
It is then clear that if the trajectories of the closed-loop system in
a neighbourhood of $\Gamma$ are bounded, the least a passivity-based
feedback can guarantee is the attractivity of $\cO$; but this is
not sufficient for our purposes. Notice that, on $\cO$,
$\varphi(\cdot)=0$ and so the closed-loop dynamics on $\cO$ coincide
with the open-loop dynamics. In particular, then, $\cO$ is an
invariant set for the closed-loop system.  In order to ensure the
property of asymptotic stability of $\Gamma$, the open-loop
system {\em must} enjoy the same property {\em relative to
  $\cO$}. Therefore, a necessary condition for $\Gamma$ to be
asymptotically stable for the closed-loop system is that $\Gamma$
be asymptotically stable relative to $\cO$ for the open-loop
system. Is this condition also sufficient or are extra-properties
  needed?  This question leads to the reduction problem stated in
the introduction: {\em If $\Gamma \subset \cO$ is stable,
  attractive, or asymptotically stable relative to $\cO$,
  what extra conditions guarantee that $\Gamma$ is stable,
  attractive, or asymptotically stable with respect to the
  whole state space?}  This problem was originally formulated by
P. Seibert and J.S. Florio in 1969-1970. Seibert and Florio developed
reduction principles for stability (see Theorem 3.4
in~\cite{SeiFlo95}) and asymptotic stability (see Theorem~4.13 and
Corollary~4.11 in~\cite{SeiFlo95}) for dynamical systems on metric
spaces assuming that $\Gamma$ is compact. Their conditions first
appeared in~\cite{Sei69} and \cite{Sei70}, while the proofs are found
in~\cite{SeiFlo95} (see also the work in~\cite{Kal99} for related
results). 

The reduction problem arises in many areas of nonlinear control
theory, including the stability of cascade-connected systems, the
separation principle of output feedback control, and the adaptive
control problem. It also plays a role in singular perturbations and
center manifold theory. The theorems below, which extend Seibert and
Florio's results in the finite dimensional setting, are relevant to
all these problems. We omit all proofs in this shortened paper, but
refer the interested reader to the full version~\cite{ElHMag09_2} 
and~\cite{ElHMag09_3}. Consider the dynamical system
\begin{equation}
\label{eq:Sigma2}
\Sigma: \ \dot x = f(x), \ x \in \cX,
\end{equation}
with $f$ locally Lipschitz on $\cX$, and let $\Gamma$ and $\cO$,
$\Gamma \subset \cO \subset \cX$, be closed sets which are positively
invariant for system $\Sigma$. We have the following

\begin{thm}[Reduction principle for
    attractivity]\label{thm:reduction:attractivity} Let $\Gamma$
  and $\cO$, $\Gamma \subset \cO \subset \cX$, be two closed positively
  invariant sets. Then, $\Gamma$ is attractive if the following
  conditions hold:

\begin{enumerate}[(i)]
\item $\Gamma$ is asymptotically stable relative to $\cO$

\item $\cO$ is locally attractive near $\Gamma$,

\item there exists a neighbourhood $\cN(\Gamma)$ such that, for all
  initial conditions in $\cN(\Gamma)$, the associated solutions are
  bounded and such that the set $\cl(\phi(\Re^+,\cN(\Gamma))) \cap \cO
  $ is contained in the domain of attraction of $\Gamma$ relative to
  $\cO$.

\end{enumerate}
The set $\Gamma$ is globally attractive if:

\begin{enumerate}[(i)']
\item $\Gamma$ is globally asymptotically stable relative to
  $\cO$,

\item $\cO$ is a global attractor,

\item all trajectories in $\cX$ are bounded.
\end{enumerate}
\end{thm}
Conditions (ii) and (ii') above are also
necessary. Theorem~\ref{thm:reduction:attractivity} is novel in that
Seibert and Florio did not investigate a reduction principle for
attractivity.

\begin{thm}[Reduction principle for asymptotic
    stability]\label{thm:reduction:asymptotic_stability:unbounded} Let
  $\Gamma$ and $\cO$, $\Gamma \subset \cO \subset \cX$, be two closed
  positively invariant sets. Then, $\Gamma$ is [globally]
  asymptotically stable if the following conditions hold:

\begin{enumerate}[(i)]
\item $\Gamma$ is [globally] asymptotically stable relative to
  $\cO$,

\item $\cO$ is locally stable near $\Gamma$,

\item $\cO$ is locally attractive near $\Gamma$  [$\cO$ is
  globally attractive], 

\item if $\Gamma$ is unbounded, then $\Sigma$ is locally uniformly
  bounded near $\Gamma$, 

\item{} [all trajectories of $\Sigma$ are bounded.]
\end{enumerate}
\end{thm}
Conditions (i), (ii), and (iii) above are necessary. 
\begin{thm}[Reduction principle for stability]
\label{thm:reduction:stability:unbounded}
Let $\Gamma$ and $\cO$, $\Gamma \subset \cO \subset \cX$, be two
closed positively invariant sets. If assumptions (i), (ii), and (iv)
of Theorem~\ref{thm:reduction:asymptotic_stability:unbounded} hold,
then $\Gamma$ is stable.
\end{thm}

If $\Gamma$ is a compact set, then
Theorems~\ref{thm:reduction:asymptotic_stability:unbounded}
and~\ref{thm:reduction:stability:unbounded} are equivalent to the
results presented in Theorems 3.4, 4.13, and Corollary 4.11
in~\cite{SeiFlo95}.

\section{Detectability}\label{sec:detectability}

For convenience, we repeat the definition of the set $\cO$ given in
Section~\ref{sec:problems}
\begin{defn}[Set $\cO$]\label{defn:O}
Given the control system~\eqref{eq:sys}, we denote by $\cO$ the
maximal set contained in $h^{-1}(0)$ which is invariant for the
open-loop system $\dot x = f(x)$.
\end{defn}
When system~\eqref{eq:sys} is linear time-invariant (LTI), the set
$\cO$ is the unobservable subspace. As discussed in
Section~\ref{sec:problems}, as long as the trajectories of the
closed-loop system in a neighbourhood of $\Gamma$ are bounded, a
passivity-based feedback renders the set $\cO$ attractive.  In
order to guarantee asymptotic stability of $\Gamma \subset \cO$,
the reduction principle in
Theorem~\ref{thm:reduction:asymptotic_stability:unbounded} suggests
that $\Gamma$ should be asymptotically stable relative to $\cO$
for the open-loop system. We call this property $\Gamma$-detectability.
\begin{defn}[$\Gamma$-detectability]\label{defn:gamma-detect}
System \eqref{eq:sys} is {\em locally $\Gamma$-detectable} if $\Gamma$
is asymptotically stable relative to $\cO$ for the open-loop
system. The system is {\em $\Gamma$-detectable} if $\Gamma$ is
globally asymptotically stable relative to $\cO$ for the
open-loop system.
\end{defn}
Our notion of detectability is parameterized by $\Gamma$, and not by
$\cO$, although the set $\cO$ figures in its definition. This is due
to the fact that $\cO$ is entirely determined by the open-loop vector
field $f$ and the output function $h$. In the case of LTI systems,
when $\Gamma = \{0\}$, the above definition requires that all
trajectories on the unobservable subspace $\cO$ converge to
$0$. Therefore, in the LTI setting, $\Gamma$-detectability coincides
with the classical notion of detectability.  Further, the notion of
$\Gamma$-detectability generalizes that of zero-state detectability
in~\cite{ByrIsiWil91}. As a matter of fact, when $V$ is positive
definite, and thus $\Gamma = \{0\}$, the two detectability notions
coincide.
\begin{lem}\label{lem:gamma_detect-zsd}
If $V$ is positive definite and $\Gamma=V^{-1}(0)=\{0\}$, then the
following three conditions are equivalent:

\begin{enumerate}[(a)]
\item System~\eqref{eq:sys} is locally zero-state detectable
  [zero-state detectable], 

\item the equilibrium $x=0$ is [globally] attractive relative to $\cO$ for the
  open-loop system,

\item system~\eqref{eq:sys} is locally $\Gamma$-detectable
  [$\Gamma$-detectable].
\end{enumerate}
\end{lem}

\begin{proof}
The set of points $x_0 \in \cX$ such that the open-loop solution
satisfies $h(\phi(t,x_0)) \equiv 0$ is precisely the maximal open-loop
invariant subset of $h^{-1}(0)$, i.e., the set $\cO$. Thus, conditions
(a) and (b) are equivalent.  Since~\eqref{eq:sys} is passive,
by~\eqref{eq:passivity} we have $L_f V \leq 0$. By the assumption that
$V$ is positive definite, it follows that $x=0$ is a stable
equilibrium of the open-loop system. Thus, $x=0$ is [globally]
asymptotically stable relative to $\cO$ for the open-loop system if
and only if $x=0$ is [globally] attractive relative to $\cO$ for the
open-loop system, proving that conditions (b) and (c) are equivalent.
\end{proof}
The next lemma shows that $\Gamma$-detectability also encompasses the
notion of $V$-detectability in~\cite{Shi00_2}.

\begin{lem}\label{lem:gamma_detect-Vdetect}
If $\Gamma = V^{-1}(0)$ is a compact set, then the following three
conditions are equivalent:

\begin{enumerate}[(a)]

\item System~\eqref{eq:sys} is locally $V$-detectable,

\item the set $\Gamma$ is attractive relative to $\cO$ for the
  open-loop system,

\item system~\eqref{eq:sys} is locally $\Gamma$-detectable.
\end{enumerate}
Moreover, if $V$ is proper, then the global versions of conditions
(a)-(c) are equivalent.
\end{lem}

\begin{proof}
Suppose that~\eqref{eq:sys} is locally $V$-detectable. Then, for
all $x_0 \in V^{-1}([0,c]) \cap \cO$, we have $V(x(t)) \to 0$. Since
$V^{-1}(0)$ is compact, in a sufficiently small neighbourhood of
$\Gamma$, $V^{-1}(\phi(t,x_0)) \to 0$ implies $\phi(t,x_0) \to
V^{-1}(0)$, and thus $\Gamma = V^{-1}(0)$ is attractive relative to
$\cO$ for the open-loop system, showing that condition (a) implies
(b).  Since $L_f V \leq 0$, $\Gamma$ is also stable for the open-loop
system. Thus, condition (b) implies (c).  Now suppose
that~\eqref{eq:sys} is locally $\Gamma$-detectable. Then,
%the condition $h(\phi(t,x_0)) \equiv 0$ is equivalent to $\phi(t,x_0)
%\in \cO$ for all $t \in \Re$. By $\Gamma$-detectability, 
there exists a neighbourhood $S$ of $\Gamma$ such that, for all $x_0
\in S \cap \cO$, $\phi(t,x_0) \to \Gamma$.  Since $\Gamma=V^{-1}(0)$
is compact and $V$ is continuous, there exists $c>0$ such that
$V^{-1}([0,c]) \subset S$. Hence, for all $x_0 \in V^{-1}([0,c]) \cap
\cO$ or, equivalently for all $x_0 \in V^{-1}([0,c])$ such that
$h(\phi(t,x_0))\equiv 0$, we have $\phi(t,x_0) \to V^{-1}(0)$. By the
continuity of $V$ and the compactness of $V^{-1}(0)$ the latter fact
implies that $V(\phi(t,x_0)) \to 0$. This proves
that condition (c) implies (a).
The proof of equivalence of the global notions of detectability
follows directly from the fact that if $V$ is proper, then
$V(\phi(t,x_0)) \to 0$ if and only if $\phi(t,x_0) \to V^{-1}(0)$.
\end{proof}
We now give sufficient conditions for~\eqref{eq:sys} to be
$\Gamma$-detectable. The proof is in~\cite{ElHMag09_2}. Let
\[
\SS = \{x \in \mathcal{X}: L_f^m h(x) =0, 0 \leq m \leq
  r+n-2\}.
\]
Notice that the definition of $\SS$ does not directly involve the
storage function (but recall that $h\trans = L_g V$, so it does
indirectly depend on $V$).

\begin{prop}
\label{prop:Gamma-detect:sufficient_conditions}
Suppose that all open-loop trajectories that originate and remain on
$\SS$ are bounded and that the open-loop system in~\eqref{eq:sys} is
locally uniformly bounded near $\Gamma$.  If
\begin{equation}\label{eq:Gamma-detect:sufficient_conditions}
\SS \cap \JJ(\SS,\SS)  \subset \Gamma,
\end{equation}
then system~\eqref{eq:sys} is $\Gamma$-detectable. Moreover, if
$\Gamma=V^{-1}(0)$, then
condition~\eqref{eq:Gamma-detect:sufficient_conditions} may be
replaced by the following one:
\begin{equation}\label{eq:Gamma-detect:sufficient_conditions:2}
\SS \cap \LL(\SS) \subset V^{-1}(0).
\end{equation}
\end{prop}

\begin{rem}
Proposition~\ref{prop:Gamma-detect:sufficient_conditions} relaxes the
sufficient conditions for detectability found in~\cite[Proposition
  3.4]{ByrIsiWil91} and~\cite[Theorem 10]{Shi00_1}. We refer the
reader to~\cite{ElHMag09_2} for a discussion. The natural way to check
$\Gamma$-detectability is to compute the set $\cO$ in
Definition~\ref{defn:O}, and then assess the asymptotic stability
of $\Gamma$ relative to $\cO$. Should the computation of the set $\cO$
be too difficult,
Proposition~\ref{prop:Gamma-detect:sufficient_conditions} above
provides an alternative, but conservative, criterion for
$\Gamma$-detectability that may prove useful in some cases.  It is
important to notice that
condition~\eqref{eq:Gamma-detect:sufficient_conditions} may be hard to
check in practice because it involves the computation of the
prolongational limit set $J^+(\SS,\SS)$. The conditions found
in~\cite[Proposition 3.4]{ByrIsiWil91} and~\cite[Theorem 10]{Shi00_1}
suffer from the same limitation because they too involve the
computation of limit sets.
\end{rem}

\section{Solution of the set stabilization
  problem}\label{sec:solution_set_stabilization} 

We are now ready to solve the stabilization problem, by presenting
conditions that guarantee that a passivity-based controller of the
form~\eqref{eq:passivity-based_feedback} makes $\Gamma$ stable,
attractive, or asymptotically stable for the closed-loop
system. All results are straightforward consequences of the reduction
principles presented in Section~\ref{sec:problems}, and they rely on
the next fundamental observation, whose proof is found
in~\cite{ElHMag09_2}.

\begin{prop}\label{prop:passivity_implies_local_stability}
Consider the passive system~\eqref{eq:sys} with a passivity-based
feedback of the form~\eqref{eq:passivity-based_feedback}, and the set
$\cO$ in Definition~\ref{defn:O}.  Then, the set $\cO$ is locally
stable near $\Gamma$ for the closed-loop system.
\end{prop}

\begin{thm}[Asymptotic stability of
    $\Gamma$]\label{thm:semi-asymptotic_stability} Consider
  system~\eqref{eq:sys} with a passivity-based feedback of the
  form~\eqref{eq:passivity-based_feedback}.  If $\Gamma$ is compact,
  then
\begin{itemize}
\item $\Gamma$ is asymptotically stable for the closed-loop system
if, and only if, system~\eqref{eq:sys} is locally
  $\Gamma$-detectable,
\item if all trajectories of the
  closed-loop system are bounded, then $\Gamma$ is globally
  asymptotically stable for the closed-loop system if, and only if,
  system~\eqref{eq:sys} is $\Gamma$-detectable.
\end{itemize}
If $\Gamma$ is unbounded and the closed-loop system is locally
uniformly bounded near $\Gamma$, then
\begin{itemize}
\item $\Gamma$ is asymptotically stable for the
  closed-loop system if, and only if, system~\eqref{eq:sys} is locally
  $\Gamma$-detectable.
\item if all trajectories of the closed-loop system are bounded,
  $\Gamma$ is globally asymptotically stable for the closed-loop
  system if, and only if, system~\eqref{eq:sys} is
  $\Gamma$-detectable.
\end{itemize}

\end{thm}

\begin{proof}
The sufficiency part of the theorem follows from the following
considerations.  By
Proposition~\ref{prop:passivity_implies_local_stability}, $\cO$ is
locally stable near $\Gamma$. If $\Gamma$ is compact, by
Theorem~\ref{thm:reduction:stability:unbounded} local
$\Gamma$-detectability implies stability of $\Gamma$. The stability
of $\Gamma$ and its compactness in turn imply that all closed-loop
trajectories in some neighbourhood of $\Gamma$ are bounded. Since all
bounded trajectories asymptotically approach $\cO$, $\cO$ is locally
attractive near $\Gamma$. If all trajectories of the closed-loop
system are bounded, then $\cO$ is globally attractive.
Theorem~\ref{thm:reduction:asymptotic_stability:unbounded} yields the
required result.

Now suppose that $\Gamma$ is unbounded. By local uniform boundedness
near $\Gamma$ we have that all closed-loop solutions in some
neighbourhood of $\Gamma$ are bounded and hence $\cO$ is locally
attractive near $\Gamma$. Once again, if all closed-loop
trajectories are bounded, then $\cO$ is globally attractive.  The
required result now follows from
Theorem~\ref{thm:reduction:asymptotic_stability:unbounded}.

The various necessity statements follow from the following basic
observation. Any passivity-based feedback of the
form~\eqref{eq:passivity-based_feedback} makes $\cO$ an invariant set
for the closed-loop system (see
Section~\ref{sec:problems}). Therefore, if $\Gamma$ is [globally]
asymptotically stable for the closed-loop system, necessarily
$\Gamma$ is [globally] asymptotically stable relative to $\cO$
for the closed-loop system. In other words,~\eqref{eq:sys} is
necessarily locally $\Gamma$-detectable [$\Gamma$-detectable].
\end{proof}
We conclude this section with the following result, which gives
conditions that are alternatives to the $\Gamma$-detectability
assumption.

\begin{prop}\label{prop:alternative_Gamma_detect}
Theorem~\ref{thm:semi-asymptotic_stability} still holds if the local
$\Gamma$-detectability [$\Gamma$-detectability] assumption is replaced
by the following condition:

\begin{enumerate}[(i')]
\item $\Gamma$ is stable relative to $V^{-1}(0)$ and $\Gamma$ is
  [globally] attractive relative to $\cO$.
\end{enumerate}
\end{prop}
We omit the proof of this proposition.  If the sufficient conditions
for $\Gamma$-detectability in
Proposition~\ref{prop:Gamma-detect:sufficient_conditions} fail, rather
than checking for $\Gamma$-detectability one may find it easier to
check condition (i') in
Proposition~\ref{prop:alternative_Gamma_detect}, because verifying
whether $\Gamma$ is stable relative to $V^{-1}(0)$ does not require
finding the maximal open-loop invariant subset $\cO$ of $h^{-1}(0)$;
moreover, checking that $\Gamma$ is attractive relative to $\cO$
amounts to checking the familiar condition in~\cite{ByrIsiWil91}
\[
h(\phi(t,x_0)) \equiv 0 \implies \phi(t,x_0) \to \Gamma \text{ as } t
\to + \infty.
\]
Note that, in the framework of~\cite{ByrIsiWil91} and~\cite{Shi00_2}, the
requirement that $\Gamma$ be stable relative to $V^{-1}(0)$ is
trivially satisfied because in these references it is assumed that
$\Gamma = V^{-1}(0)$.

\section{Discussion}\label{sec:discussion}

Theorem 3.2 in~\cite{ByrIsiWil91} and Theorem 2.3 in~\cite{Shi00_2},
dealing with the special case when $\Gamma=V^{-1}(0)$ $(=\{0\})$ and
$\Gamma$ is compact, become corollaries of our main result,
Theorem~\ref{thm:semi-asymptotic_stability}.  We have already shown
(see Lemmas~\ref{lem:gamma_detect-zsd}
and~\ref{lem:gamma_detect-Vdetect}) that in this special case the
properties of zero-state detectability (when $\Gamma = \{0\}$), and
$V$-detectability coincide with our notion of
$\Gamma$-detectability. In this context, then, the results
in~\cite{ByrIsiWil91} and~\cite{Shi00_2} state that local
$\Gamma$-detectability is a sufficient condition for the asymptotic
stabilization of the origin using a passivity-based feedback. We have
shown that actually this condition is also {\em necessary.} An
analogous remark can be made for the global solution of the set
stabilization problem.

The theory in~\cite{ByrIsiWil91} and~\cite{Shi00_2} does not handle
the special case when $\Gamma$ is compact and $\Gamma\subsetneq
V^{-1}(0)$, while our theory does. This case includes the important
situation when one wants to stabilize an equilibrium ($\Gamma=\{0\}$)
but the storage is only positive semi-definite. Based on the results
in~\cite{ByrIsiWil91} and~\cite{Shi00_2}, it may be tempting to
conjecture that Theorem 3.2 in~\cite{ByrIsiWil91} and Theorem 2.3
in~\cite{Shi00_2} still hold if one employs the following notion of
detectability:
\begin{equation}\label{eq:detectability_conjecture}
\begin{aligned}
(\forall x_0 \in \cN(\Gamma) ) \ h(\phi(t,x_0)) = 0 &\text{ for all } t
\in \Re\\
&\implies \phi(t,x_0) \to \Gamma,
\end{aligned}
\end{equation}
which corresponds to requiring that $\cO$ in Definition~\ref{defn:O}
is an attractor for the open-loop system.  This conjecture is
false: we have shown that (local) $\Gamma$-detectability (i.e., the
asymptotic stability of $\Gamma$ relative to $\cO$ for the
open-loop system) is a necessary condition for the stabilization of
$\Gamma$. Even if one relaxes the asymptotic stability requirement and
just asks for attractivity of $\Gamma$ relative to $\cO$, the
above conjecture is still false. As a matter of fact,
Theorem~\ref{thm:reduction:attractivity} suggests that even in this
case (local) $\Gamma$-detectability is a key property. A
counter-example illustrating this loss of attractivity is the
pendulum.  The upright equilibrium is globally attractive, but
unstable, relative to the homoclinic orbit of the pendulum. Despite
the fact that a passivity-based feedback can be used to asymptotically
stabilize the homoclinic orbit (see,
e.g.,~\cite{Fra96},~\cite{AstFur00}), the upright equilibrium is
unstable for the closed-loop system. This well-known phenomenon finds
explanation in the theory developed in this paper: the cause of the
problem is the instability of the upright equilibrium relative to the
homoclinic orbit.  We next present another explicit counter-example
illustrating our point.

\begin{figure*}[tb]
\psfrag{x1}{{\small $x_1$}}
\psfrag{x2}{{\small $x_2$}}
\psfrag{x3}{{\small $x_3$}}
\psfrag{G}{{\small $\Gamma$}}
\psfrag{O}{{\small $\cO$}}
\centerline{\epsfig{figure=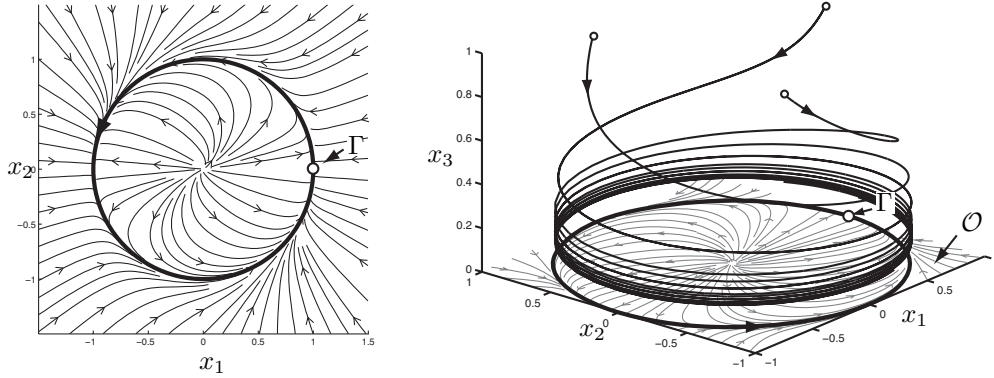,width=.8\textwidth}}
\caption{On the left-hand side, phase portrait on $\cO$ for the
  open-loop system~\eqref{eq:example:open-loop_on_O}. $\Gamma$ is
  globally attractive relative to $\cO$. On the right-hand side,
  closed-loop system~\eqref{eq:example:sys} with feedback
  $u=-y$. $\Gamma$ is {\em not} attractive relative to the whole state
  space.}
\label{fig:closed-loop_attractivity}
\end{figure*}

\begin{example}
Consider the  control system with state $(x_1,x_2,x_3)$,
\begin{equation}\label{eq:example:sys}
\begin{aligned}
&\dot r = -r(r-1) \\
&\dot \theta = \sin^2(\theta/2)+x_3\\
&\dot{x}_3=u \\
&y=x_3^3,
\end{aligned}
\end{equation}
where $(r,\theta) \in (0,+\infty) \times S^1$ represent polar
coordinates for $(x_1,x_2)$. The control system is passive with
storage $V(x) = x_3^4/4$. Let $\Gamma$ be the equilibrium point
$\{(x_1,x_2,x_3): x_1=1, x_2=x_3=0\}$ and note that
$\cO=\{(x_1,x_2,x_3):x_3=0\}$. On $\cO$, the open-loop dynamics read as
\begin{equation}\label{eq:example:open-loop_on_O}
\begin{aligned}
&\dot r = -r(r-1) \\
&\dot \theta = \sin^2(\theta/2),
\end{aligned}
\end{equation}

\noindent
and it is easily seen that the equilibrium $\Gamma$ attracts every
point in $\cO$ except the origin. Hence, $\Gamma$ is attractive
relative to $\cO$, but unstable (indeed, the unit circle is a
homoclinic orbit of the equilibrium); see
Figure~\ref{fig:closed-loop_attractivity}. Therefore,
condition~\eqref{eq:detectability_conjecture} holds but the system is
not locally $\Gamma$-detectable. Consider the passivity-based feedback
$u = -y$, which renders $\cO$ globally asymptotically stable. Now for
any initial condition off of $\cO$ such that $(x_1(0),x_2(0)) \neq
(0,0)$, $x_3(0)>0$, the corresponding trajectory is bounded, but its
positive limit set is the unit circle on $\cO$, and therefore it is
not a subset of $\Gamma$; see
Figure~\ref{fig:closed-loop_attractivity}. In conclusion, $\Gamma$ is
not attractive for the closed-loop system (and neither is it stable).
This example illustrates the fact that, when $\Gamma \subsetneq
V^{-1}(0)$ is compact, simply requiring
condition~\eqref{eq:detectability_conjecture} in place of
$\Gamma$-detectability may not be enough for attractivity of $\Gamma$.
\end{example}
In the light of Theorem~\ref{thm:semi-asymptotic_stability} and the
example above, it is clear that the addition of the stability
requirement on $\Gamma$, relative to $\cO$, is a crucial enhancement
to the notions of detectability in~\cite{ByrIsiWil91} and~\cite{Shi00_2}.

\bibliographystyle{IEEEtran}        
\bibliography{tac}

\end{document}